\documentclass[12pt, a4paper]{amsart}

\usepackage{ifthen}
\usepackage{tikz}

\usepackage{dsfont} 

\newtheorem*{rep@theorem}{\rep@title}
\newcommand{\newreptheorem}[2]{%
\newenvironment{rep#1}[1]{%
 \def\rep@title{#2 \ref{##1}}%
 \begin{rep@theorem}}%
 {\end{rep@theorem}}}
\makeatother

\usepackage{breqn}
\usepackage{amscd,amsmath,amssymb,amsthm,amsfonts}
\usepackage[alphabetic]{amsrefs}
\usepackage{mathrsfs}
\usepackage[shortlabels]{enumitem}
\usepackage[all]{xy}

\usepackage{thmtools}

\usepackage{xcolor}
\usepackage[hypertexnames=true,colorlinks = true, allcolors=blue,linktoc = all, pdffitwindow = false, urlbordercolor = white]{hyperref}%

\usepackage{hyperref}
\usepackage{graphicx} 

 \newtheorem{lettertheorem}{Theorem}

\def\ad{\operatorname{ad}}
\def\antidiag{\operatorname{antidiag}}

\def\End{\operatorname{End}}

\def\supp{\operatorname{supp}}
\def\ker{\operatorname{ker}}

\def\dim{\operatorname{dim}}
\def\antidiag{\operatorname{\nabla}}
\def\diag{\operatorname{\Delta}}

\def\id{\operatorname{id}}

\def\Im{\operatorname{Im}}

\def\id{\operatorname{id}}
\def\supp{\operatorname{supp}}

\def\Spec{\operatorname{Spec}}

\def\Irr{\operatorname{Irr}}

\def\C{\mathbb{C}}

\def\T{\mathbb{T}}

\newcommand{\IC}[0]{\mathbb{C}}

 \newcommand{\IN}[0]{\mathbb{N}}
 
 \newcommand{\IR}[0]{\mathbb{R}}
 \newcommand{\IT}[0]{\mathbb{T}}

 \newcommand{\IZ}[0]{\mathbb{Z}}

 \newcommand{\CB}[0]{\mathcal{B}}
 
 \newcommand{\CF}[0]{\mathcal{F}}
 \newcommand{\CH}[0]{\mathcal{H}}
 
\newcommand{\CK}[0]{\mathcal{K}} 
 
\newcommand{\CO}[0]{\mathcal{O}} \newcommand{\CP}[0]{\mathcal{P}}
 
 \newcommand{\CT}[0]{\mathcal{T}}

\DeclareMathOperator{\Rep}{Rep}

\newcommand{\TL}[2]{%
\vcenter{\hbox{\begin{tikzpicture}
\foreach \x/\y/\z/\w in {#2} {
  \ifthenelse{\x = \z \AND \y = \w}
  {}
  {
    \ifthenelse{\y = \w}
    {\ifthenelse{\y = 0}
      {\draw (\x, \y) .. controls +(0, 0.5) and +(0, 0.5) .. (\z, \w);}
      {\draw (\x, \y) .. controls +(0, -0.5) and +(0, -0.5) .. (\z, \w);}
    }
    {\draw (\x, \y) .. controls +(0, 1) and +(0, -1) .. (\z, \w);}
  }
}
\foreach \x/\y/\z/\w in {#2} {
  \fill (\x, \y) circle (2pt); 
  \fill (\z, \w) circle (2pt); 
}
\draw (-0.5, 0) rectangle ({#1 - 0.5}, 1);
\clip (-0.6, 0) rectangle ({#1 - 0.6}, 1);
\end{tikzpicture}
}}}

\newreptheorem{theorem}{Theorem}
\newtheorem{theorem}{Theorem}[section]
\newtheorem*{theorem*}{Theorem}
\newtheorem*{proposition*}{Proposition}
\newtheorem{proposition}[theorem]{Proposition}
\newtheorem{lemma}[theorem]{Lemma}
\newtheorem*{lemma*}{Lemma}
\newtheorem{example}[theorem]{Example}

\newtheorem{definition}[theorem]{Definition}

\newtheorem{remark}[theorem]{Remark}

\newtheorem*{notations}{Notations}

\numberwithin{equation}{section}

\begin{document}


 \title[The simplest Motzkin subproduct]{The Cuntz-Pimsner algebra of the simplest Motzkin subproduct system is $2$-subhomogeneous}
\author{Valeriano Aiello} 
\address{Valeriano Aiello, Dipartimento di Matematica, Universit\`a di Roma La Sapienza, P.le Aldo Moro 5, 00185 Roma, Italy 
}
\email{valerianoaiello@gmail.com} 
\author{Arnaud Brothier} 
\address{Arnaud Brothier,  University of Trieste, Department of Mathematics,
via Valerio 12/1, 34127, Trieste, Italy and School of Mathematics and
Statistics, University of New South Wales, Sydney NSW 2052, Australia, 
\url{https://sites.google.com/site/arnaudbrothier/}
}
\email{arnaud.brothier@gmail.com} 

\begin{abstract} 
The Motzkin subproduct system (SPS) is 
constructed from the Jones-Wenzl idempotents of the Motzkin algebras, which generalizes  the Temperley-Lieb SPS.
The simplest  Motzkin SPS, which is not a Temperley-Lieb SPS, is constructed from a $3$-dimensional Hilbert space.
We explicitly describe the spectrum of the corresponding Cuntz–Pimsner algebra, which remarkably admits only irreducible representations of dimensions \(1\) and \(2\), and can thus be viewed as a mildly quantum space.
Moreover, we  analyse representations of the Richard Thompson groups and of the Cuntz algebra that are associated to this spectrum.
\end{abstract}

\maketitle

\section*{Introduction}

The full Fock space over an \(n\)-dimensional Hilbert space has been, either directly or indirectly, a source of interesting   C\(^*\)-algebras. 
Indeed, one defines the usual creation and annihilation operators, which, together with the identity, generate the Toeplitz(-Cuntz) C\(^*\)-algebra \cite{Cuntz}.
This algebra contains the compact operators as a closed ideal and the corresponding quotient is isomorphic to the Cuntz(-Dixmier) algebra \(\mathcal{O}_n\) \cite{Dixmier64,Cuntz}.

A seminal generalization of this picture is due to Arveson \cite{Arveson98}, who replaced the full Fock space with the subspace consisting of symmetric tensors. 
The corresponding compressions of the creation and annihilation operators, together with the identity, generate a new Toeplitz algebra that again contains the compact operators as an ideal. In this case, however, the quotient  turns out to be the C\(^*\)-algebra of continuous functions on a sphere.

This framework was subsequently extended to the setting of subproduct systems. These are collections of finite-dimensional Hilbert spaces  
$\{H_k\}_{k\geq 0}$
satisfying certain compatibility axioms \cites{Bhat, Shalit}. Such systems give rise to closed subspaces of the full Fock space, and hence to associated orthogonal projections. The corresponding compressed creation and annihilation operators generate a Toeplitz algebra containing 
an ideal consisting of compact operators, and the quotient is known as a Cuntz–Pimsner algebra.

In the spirit of the noncommutative algebraic geometry program initiated in \cites{Popescu, Popescu2}, subproduct systems admit a purely algebraic interpretation: there is a one-to-one correspondence between subproduct systems and homogeneous ideals in the free algebra \(\mathbb{C}\langle X_1, \dots, X_n\rangle\), providing a noncommutative analogue of the classical Nullstellensatz, \cite{Shalit}*{Proposition 7.2}.

Arveson's original subproduct system carried a natural \(U(2)\)-action, making it equivariant. Later, Arici and Kaad \cite{AK} introduced and studied a family of \(SU(2)\)-equivariant subproduct systems. 
Another source of examples arises from Jones's planar algebras and in particular the Temperley-Lieb(-Jones) algebras \cite{Jones83,jo2}. 
First, Habbestad and Neshveyev constructed subproduct systems from the Temperley-Lieb algebras and specifically from the ranges of Jones–Wenzl idempotents in suitable representations \cites{Nesh, Nesh2, Nesh3}. These subproduct systems are called 
Temperley-Lieb subproduct systems.
A key novelty  is that they admit quantum group symmetries.

The Motzkin algebras are natural generalisations of the Temperley–Lieb algebras \cite{Halverson}.
A new family of subproduct systems built from Motzkin algebras was introduced in \cite{ADR01}. 
While this construction recovers the earlier Temperley–Lieb systems as a special case, the presence of a quantum group symmetry remains an open question. 
Both the Temperley–Lieb and Motzkin subproduct systems yield universal Toeplitz and Cuntz–Pimsner algebras. 
Notably, all the aforementioned examples correspond to homogeneous ideals generated by a single quadratic polynomial.

In this article we study the lowest dimensional Motzkin subproduct system which is not a Temperley-Lieb subproduct system. We refer to this object as the \emph{simplest Motzkin subproduct system}. 

Recall that a \(C^*\)-algebra is called \(n\)-subhomogeneous if the dimensions of its irreducible representations are uniformly bounded above by \(n\). Our main result, stated below, describes the fine structure of the Cuntz–Pimsner algebra attached to the simplest Motzkin subproduct system: it is 2-subhomogeneous.
This algebra occupies an intermediate position in the hierarchy of examples.
Indeed, our algebra is not commutative, unlike Arveson spherical quotient algebras but does not carry sophisticated quantum group structure like the algebra of Habbestad and Neshveyev.
To set the notations, we denote by $\IT$ the set of complex numbers of modulus $1$.
\begin{lettertheorem}\label{maintheo}
The Cuntz-Pimsner algebra $\CO_P$ of the simplest Motzkin subproduct system is $2$-subhomogeneous.
More precisely, the spectrum $\Spec$ of $\CO_P$ decomposes into $\Spec(1)\sqcup \Spec(2)$ where the 1-dimensional part of the spectrum $\Spec(1)$ is homeomorphic to the disjoint union of two circles
and $\Spec(2)$ is homeomorphic to the 3-dimensional object
$$\dfrac{[0,1/\sqrt 2]\times \IT^2\setminus (C_+\cup C_-)}{(0,u,v)\sim (0,u,\hat v) \text{ and } (1/\sqrt 2, u,v)\sim (1/\sqrt 2, \overline u v^2,v)},$$
where $C_\pm$ is the circle $\{(1/\sqrt 2, u,\pm u):\ u\in\IT\}.$
\end{lettertheorem}
We observe that the whole spectrum is not Hausdorff: the net $(p,u^2, \pm u^2)$, with $p\nearrow 1/\sqrt{2}$, has two limit points, see Section \ref{sec:proof-main} for details.

Finally, a functorial procedure allows one to define a representation of the (usual) Cuntz algebra $\CO_2$ from a pair of operators $A,B$ satisfying $A^*A+B^*B=1$ and thus, by restriction, unitarily representations of the Richard Thompson groups $F,T,V$, see \cite{Brothier-Jones19}.
In particular, any representation of $\CO_P$ defines a representation of $F,T,V$, and $\CO_2.$
We briefly analyse in Section \ref{sec:FTVO} the representations of $F,T,V,\CO_2$ obtained from the spectrum of $\CO_P.$

\section{Cuntz-type algebras of Motzkin subproduct systems}
\subsection{Motzkin algebras}

In this section, we first recall the definition and some key properties of the Motzkin algebras, then we define some representations 
of them. For more detailed information, refer to \cites{Halverson, Ly, Jo21}.

For any real number $\lambda\neq 0$, define the sequence of Motzkin algebras $M_k(\lambda^{-1})$ with $k\geq 1$ a natural number as follows.
For $k=1$,  $M_1(\lambda^{-1})$ is the unital universal $*$-algebra generated by a self-adjoint idempotent $p_1$.  
For $k\geq 2$, $M_k(\lambda^{-1})$ is the universal *-algebra generated by $1$, $t_1$, \ldots, $t_{k-1}$, $l_1$, \ldots, $l_{k-1}$, satisfying the following relations \cite{Ly}*{Theorem 4.1}
\begin{enumerate}[start = 0] \label{relations-motzkin}
\item $t_i=t_i^*$,
\item $l_i^2=l_i^3$,
\item $l_i l_{i+1} l_i = l_i l_{i+1} = l_{i+1} l_i l_{i+1}$,
\item $l_{i} l_{i}^* l_{i} = l_{i}$,
\item $l_{i+1} l_{i}^* l_{i} = l_{i+1} l_{i}^* $, $l_{i} l_{i}^* l_{i-1} = l_{i}^* l_{i-1}$,
\item $l_{i} l_i^* = l_{i+1}^* l_{i+1}$,
\item if $|i-j|\geq 2$, one has $l_{i}^* l_{j} = l_j l_i^*$, $l_il_j=l_jl_i$,  $t_it_j=t_jt_i$,
$l_i t_j=t_jl_i$, $l_i^* t_j=t_jl_i^*$, 
\item $t_i^2=t_i$,
\item $t_it_{i+1}t_i=\lambda^{2} t_i$, $t_{i+1}t_{i}t_{i+1}=\lambda^{2} t_{i+1}$, 
\item $t_il_i = t_i l_i^*$,
\item $\lambda t_i l_{i+1}^* = t_i t_{i+1} l_i$,
\item $l_i^* l_{i+1}^* t_i = t_{i+1} l_i^* l _{i+1}^*$,
\item $t_i l_i t_i = \lambda t_i$.
\end{enumerate}
The dimension of  $M_k$ is the $(2k)$-th Motzkin number, whence the name, \cite{Halverson}.
 The elements $1$, $t_1$, \ldots, $t_{k-1}$, along with the relations in (0), (6), (7), (8), provide a presentation of the Temperley-Lieb algebra TL$_k(\lambda^{-1})$. 
Similarly to the Temperley-Lieb algebra, the Motzkin algebra has its 
  own Jones-Wenzl idempotents $g_k\in M_k(\lambda^{-1})$, 
  \cite{Jo21}*{Section 2.4}
  (see also \cite[Section 2]{ADR01} for additional properties of these idempotents). 
\begin{definition}
 Let $P_m(x)$, $n\geq 0$,  be the Chebyshev polynomials over $\IC$ defined as $P_0(x)=P_1(x)=1$, $P_{m+1}(x)=P_m(x)-xP_{m-1}(x)$. 
$\lambda^{-1}\in \IC$ is said to be $n$-generic if $P_k((\lambda^{-1}-1)^{-2})\neq 0$ for $1\leq k \leq n$ and generic if $P_k((\lambda^{-1}-1)^{-2})\neq 0$ for all $k\in\IN_0$.
If $\lambda^{-1}$ $n$-generic, the Motzkin Jones-Wenzl idempotents $g_k\in M_k(\lambda^{-1})$, for $1\leq k \leq n-1$, are defined inductively as 
\begin{align*}\label{recurrence-formula}
g_{k+1}&=g_k(1-p_{k+1})-\frac{\lambda^{-1}}{\lambda^{-1}-1} \frac{P_{k-1}((\lambda^{-1}-1)^{-2}))}{P_k((\lambda^{-1}-1)^{-2}))} g_k t_kg_k\,,
\end{align*}
where $g_1=1-p_1$.
\end{definition}

Note that Motzkin algebras enjoy a practical diagrammatic description just like Temperley-Lieb algebras \cite{Jo21}. 
However, we will not use it in this article.
 
\subsection{C*-algebras from Motzkin algebra's representations}

This section is devoted to introducing the Motzkin subproduct system, which was defined in \cite{ADR01}. In order to do this, we start by recalling the notion of standard 
subproduct system \cite{Shalit}*{Section 1, Definition 1.1} for the monoid $\IZ_+$.

 \begin{definition}\label{defsub}
A  subproduct system (over the additive monoid $\IZ_+$) is a sequence of Hilbert spaces $\{H_k\}_{k\geq 0}$ such that $H_0=\IC$, 
$\dim H_1<+\infty$, along with a family of 
co-products, namely, 
isometries $w_{k,l}: H_{k+l}\to H_k\otimes H_l$ satisfying
the following co-associativity condition
$$
(w_{k,l}\otimes 1)w_{k+l,n}= (1\otimes w_{l,n})w_{k,l+n}\qquad k, l, n\in \IZ_+\, . 
$$
Here the maps $w_{0,n}$  and $w_{n,0}$ are the canonical identifications.
\end{definition}

By \cite{Shalit}*{Lemma 6.1}, we do not harm generality in assuming that $H_0=\IC$, 
$H_p\subset (H_1)^{\otimes p}$,
$H_{p+q}\subset H_p\otimes H_q$,
and the isometries $w_{k,l}$ are the embedding maps. Subproduct systems realized in this way 
are known as  standard subproduct systems.
If we consider the projections $f_n: (H_1)^{\otimes n}\to H_n\subset (H_1)^{\otimes n}$, then the co-associativity condition can be recast as $f_{n+1}$ being smaller (as a projection) than $f_n\otimes 1$ and $1\otimes f_n$, see  \cite{GerholdSkeide}*{Theorem 7.2}.
 
Starting with a finite-dimensional Hilbert space $H$ of dimension $n$,
we can consider a standard subproduct system  with $H_1=H$.
We may then construct the associated Fock space $\CF_\CH :=\oplus_{k\geq 0} H_k$, 
 as a subspace of the full Fock space $\CF(H):=\oplus_{k\geq 0} H^{\otimes k}$.
To this end, we fix  an orthonormal basis $\{v_i\}_{i=1}^n$ of $H_1$. 
Denote by $e_\CH$ the orthogonal projection of $\CF(H)$ onto $\CF_\CH$.
Fix an orthonormal basis $\{v_i\}_{i=1}^n$ of $H_1$.
For any $v\in H_1$, define the creation operators $T_v: \CF(H)\to \CF(H)$ and
$S_v: \CF_\CH\to \CF_\CH$ by 
\begin{equation*}
\begin{aligned}
&T_v u := v \otimes u, \qquad v \in H_1, \; u \in \CF(H)\\
&T_i:=T_{v_i}\\
&S_i:= e_\CH T_i \upharpoonright_{\CF_\CH}\\
&S_i^* = e_\CH T_i^* \upharpoonright_{\CF_\CH}=T_i^* \upharpoonright_{\CF_\CH}\,.
\end{aligned}
\end{equation*} 

Note that 
$$
1-\sum_{i=1}^n S_iS_i^* = e_0\,,
$$ 
where $e_0$ is the orthogonal projection of  $\CF_\CH$ onto $H_0$. 
\begin{definition}
We denote by $\CT_\CH$ the unital C$^*$-algebra generated by $\{S_v$ : $v\in H_1\}$, which is known as the
 Toeplitz algebra 
$\CT_\CH$
associated with $\CH$.

As $1\in H_0$ is cyclic for $\CT_\CH$,
 the ideal of compact operators $\CK(\CF_\CH)$ is contained in  $\CT_\CH$.
The corresponding quotient 
 $\CO_\CH:=\CT_\CH/\CK(\CF_\CH)$ is the so-called    Cuntz-Pimsner algebra.  
 \end{definition}
When the subproduct system consists of  $H_k=H^{\otimes k}$, $\CO_\CH$ is the celebrated Cuntz algebra (with $\dim(H)$ generators), \cite{Cuntz}. 
Recall that this C*-algebra was first discovered by Dixmier and was later extensively studied by Cuntz \cite{Dixmier64,Cuntz}. We refer the reader to the survey \cite{ACRsurvey} for an overview of the research directions developed over the past decades.
Cuntz-Pimsner algebras for subproduct systems were introduced by Viselter in \cite{Viselter12}, generalizing the corresponding concept  for C$^*$-correspondences.

In this paper, we study the subproduct system given by $H_0=\IC$ and $H_k:=g_{k}H^{\otimes k}$ for $k\geq 1$, where the Motzkin algebra $M_k(\lambda^{-1})$ is thought of as being realised concretely in a representation.

\begin{center}
\textbf{From now on we restrict our analysis to $\lambda\in (0,1/3]$, where the Motzkin algebra affords a C$^*$-algebraic structure,  see \cite{Jo21}.}
\end{center}
In order to make formulas more readable, 
we will often adopt an involution map on the finite set $\{1,\ldots, n\}$ defined
as $\bar{i}:=n-i+1$ for any $i$.

\begin{definition}\label{Motzkin_pair} \cite{ADR01}
Let $H$ be a Hilbert space, with $\dim H=n\geq 2$, and fix an orthonormal basis $\{v_i\}_{i=1}^n$.
A Motzkin pair is given by $(v_A, v)$, where $A\in M_{n}(\IC)$ has the form $Av_i=a_i v_{\bar i}$, 
 $v_A:=\sum_i v_i\otimes Av_i\in H\otimes H$ is a (unit) vector  
 such that $\lambda = |a_i a_{\bar i}|=\bar{a}_i a_{\bar i}$ for all $i\in \{1, \ldots, n\}$, $v=\sum_{i=1}^n b_iv_i$ is a unit vector of $H$,  
 satisfying the following conditions
\begin{equation*}\label{eqMotzkinPair}
\begin{aligned}
&\lambda = \bar{a}_i a_{\bar i}\in\IR_+\, , \qquad
\sum_{i=1}^n |a_h|^2= \sum_{i=1}^n |b_h|^2=1\, , \qquad
  \overline{a }_j \bar b_i  b_{\bar j} =     \bar{a}_{\bar i} \bar b_j b_{\bar i} 
\end{aligned}
\end{equation*}
where $1\leq i, j\leq n$.\\
Every Motzkin pair yields a representation 
$\pi_{v_{A},v}: M_k(\lambda^{-1})\to \CB(H^{\otimes k})$ by setting
for all  $1\leq i, j\leq n$
\begin{equation*}
\begin{aligned}
&p(x)= \langle x,v\rangle v \, ,\\
&t(x\otimes y)= \langle x\otimes y, v_A \rangle v_A\, ,\\
&l(v_i\otimes v_j)= v_j\otimes p(v_{i})\, , \qquad l^*(v_i\otimes v_j)=  p(v_{j})\otimes v_i\, ,
\end{aligned} 
\end{equation*}
and  
\begin{equation*}
\begin{aligned}
&\pi_{v_A,v}(t_i) = \id^{\otimes (i-1)} \otimes t \otimes \id^{\otimes (k-i-1)} & \qquad 1 \leq i \leq k-1 \\
&\pi_{v_A,v}(l_i) = \id^{\otimes (i-1)} \otimes l \otimes \id^{\otimes (k-i-1)} & \qquad  1 \leq i \leq k-1 \\
&\pi_{v_A,v}(p_i) = \id^{\otimes (i-1)} \otimes p \otimes \id^{\otimes (k-i)}  & \qquad  1 \leq i \leq k
\end{aligned}
\end{equation*}
where $\id\in\CB(H)$.
\end{definition} 
\begin{remark}
 The vector $v_A$ introduced above is the so-called \emph{Temperley--Lieb vector}, originally defined in \cite[Definition~1.2]{Nesh}. We point out two minor differences between the conventions adopted in \cite{Nesh} and those used in \cite{ADR01} and in the present paper. First, our vector is normalized. Second, if the parameter of  the Temperley-Lieb algebra of \cite{Nesh} 
is denoted by $\theta$, our parameter $\lambda$ is related to theirs by
$\lambda=\theta^{-2}$. 
\end{remark}

We now recall three families of examples of Motzkin pairs from \cite{ADR01}*{Example 1.7-(iii), p.13}, the third being   the most general.
\begin{example}\label{seriesofexamples}
 \begin{enumerate}
\item[i)] If $\dim H=n=2m+1$,
a Motzkin pair is provided by 
a matrix $A$ such that $\bar{a}_ia_{\bar i}=\lambda$ for all $i\in \{1, \ldots, n\}$, and a vector $v=v_{m+1}$.  
Note that for $n=3$, 
this representation coincides with that   
of \cite{Halverson}*{Section 3.4} on $(\id-p)^{\otimes k}H^{\otimes k}$, where the $k$-th Motzkin algebra is realised as
$\End_{U_q(\mathfrak{gl}_2)}(H^{\otimes k})$.
\item[ii)] In this case we consider a vector
  $v=\sum_{i=1}^n b_i v_i$ with 
$b_i=  b_{\bar{i}}\in\IR$ for $1\leq i\leq n$ and a matrix $A$ with
 $a_i=a_{j}$ for all $i, j\in \supp(v)$.
For example, one can choose
$v= (\sqrt{2})^{-1}v_1+ (\sqrt{2})^{-1}v_n$,  and   
any $a_1$ of absolute value $\sqrt\lambda$. 
\item[iii)] One may also choose vectors with larger supports, such as, for any $r\leq n/2$, a solution is given by $b_i:=(\sum_{j=1}^r \delta_{i,j}+ \delta_{i,\bar{j}})/\sqrt{2r}$ with $a_1=\ldots =a_r=a_{\bar{1}}=\ldots =a_{\bar{r}}$ and $|a_1|=\sqrt\lambda$.
\end{enumerate}
 \end{example}  

\begin{remark}
Up to a change of basis and possibly changing one entry of the matrix defining the Temperley-Lieb vector, we may recover the subproduct systems 
of \cite{Nesh}.
Indeed, in Example \ref{seriesofexamples}-(i)  gives back the Temperley-Lieb subproduct system of \cite{Nesh} produced by an 
even dimensional Hilbert space,
when  $\dim H$ is odd, i.e. $\dim H_1$ even.
Likewise Example \ref{seriesofexamples}-(ii) returns the Temperley-Lieb subproduct system of \cite{Nesh}
produced by an odd dimensional Hilbert space,
when $\dim H$ is even, namely $\dim H_1$ is odd. 
However, if  $\dim H$ is odd, the subproduct system of Example \ref{seriesofexamples}-(ii) is novel.
\end{remark}

\subsection{The case of interest}
In this article we will focus on the simplest Motzkin pair of 
Example \ref{seriesofexamples}-(iii), which is not a Temperley-Lieb subproduct system.
More precisely, we choose
$\dim H=3$ and $r=1$ (which is the same as Example \ref{seriesofexamples}-(ii) with $\dim H=3$). 
Since $1=|a_1|^2+|a_2|^2+|a_3|^2=3\lambda$, this forces $\lambda =1/3$.
As $\lambda^{-1} -1= q+q^{-1}$, $q>0$, this leads to $q=1$. With these parameters, according to \cite{ADR01}*{Corollary 3.11, p.31} the Cuntz-Pimsner algebra $\CO_P$ is generated by two elements $s_1$, $s_2$ such that
\begin{equation}\label{first_version_relations}
\begin{cases}
    & s_1s_1^*+s_2s_2^*=1,\\
& -a_1s_1^2+a_2s_2^2=0,\\
&s_i^*s_i=1-s_is_i^*, \qquad i\in\{1, 2\}, \\
&s_2^*s_1=3\bar{a}_1 a_2s_2s_1^*. 
\end{cases}
\end{equation}
Set $z_i:=\sqrt{3}\ a_i$ and $w:=\bar z_1 z_2$ noting that $z_i,w$ must have modulus 1.
The relations become
\begin{equation}\label{second_version_relations}
\begin{cases}
    & s_1s_1^*+s_2s_2^*=1,\\
& s_1^2=w\ s_2^2,\\
&s_i^*s_i=1-s_is_i^*, \qquad i\in\{1, 2\}, \\
&s_2^*s_1=w\ s_2s_1^*. 
\end{cases}
\end{equation}
Choose a square root $\sqrt{w}$ of $w$ and put $t_1:=s_1$ and $t_2:=\sqrt{w}\ s_2.$
Since $\sqrt{w}$ has modulus 1, the first and third equations remain unchanged: 
$$t_1t_1^*+t_2t_2^*=1 \text{ and } t_i^*t_i+t_it_i^*=1 \text{ for } i\in\{1,2\}.$$
Though, we obtain
$$t_1^2=w\ s_2^2=(\sqrt{w}\ s_2)^2=t_2^2$$
and
$$t_2^*t_1=(\sqrt{w} s_2)^* s_1= \overline{\sqrt{w}}\ s_2^* s_1=\overline{\sqrt{w}}\ w s_2 s_1^*= (\sqrt{w}\ s_2) s_1^*=t_2t_1^*.$$
We obtain a presentation similar to the one above (with $t_1,t_2$ generators) but where all the constants $a_1,a_2,3\bar a_1a_2$ become $1$.
Rather than taking $t_1,t_2$ for generators we take their adjoint $a:=t_1^*$ and $b:=t_2^*$.
This choice of notation is motivated by the so-called Pythagorean algebra $\CP_2:=C^*\langle a,b| a^*a+b^*b=1\rangle$ defined in \cite{Brothier-Jones19}.
A remarkable lifting process allows one to functorially take in input a representation of $\CP_2$ and to give in output a representation of the Cuntz algebra $\CO_2$.
Our Cuntz-Pimsner algebra $\CO_P$ is a quotient of $\CP_2$ (under the obvious map) since its presentation is given by the one of $\CP_2$ and three additional relations.

In the last section we will interpret the representations of $\CO_2$ obtained from representations of $\CO_P$ and the functorial lifting process mentioned above.

\begin{notations} 
In the sequel we adopt the notation. Lowercase letters are employed for the generators of universal C$^*$-algebras, while uppercase letters are reserved for
the images of those generators in a representation. 
The C*-algebra $\CO_P$ is then generated by $a,b$ subject to the conditions
\begin{equation}\label{second_version_relations}
\begin{cases}
&1=a^*a+b^*b=a^*a+aa^*=b^*b+bb^*,\\
& a^2=b^2,\\
&ba^*=b^*a. 
\end{cases}    
\end{equation}
\end{notations}

\section{Computation of the spectrum}

Consider the C*-algebra $\CO_P$ generated by $a$, $b$ subject to the relations \eqref{second_version_relations}.
A representation of $\CO_P$ is a unital continuous *-algebra morphism $\pi:\CO_P\to B(H)$, where $B(H)$ is the algebra of bounded linear operators acting on a complex Hilbert space $H$ with adjoint as *-structure.
Representations are in one-to-one correspondences with triples $(A,B,H)$ where $H$ is a Hilbert space and $A,B\in B(H)$ satisfy \eqref{second_version_relations}.

A representation $\pi:\CO_P\to B(H)$ is \emph{reducible} if there exists a non-zero proper Hilbert subspace $K\subset H$ that is closed under the action of $\CO_P.$
It is \emph{irreducible} if it is non-zero and not reducible. 
Equivalently, $\pi:\CO_P\to B(H)$ is irreducible if $H\neq \{0\}$ and if the von Neumann algebra generated by $\pi(A)$ and $\pi(B)$ is equal to $B(H).$

Note that all irreducible representations of $\CO_P$ are separable since $\CO_P$ is separable. 
Two representations $\pi:\CO_P\to B(H)$ and $\pi':\CO_P\to B(H')$ are \emph{unitarily conjugate} if there exists a unitary transformation $U:H\to H'$ satisfying $U\pi(x)U^*=\pi'(x)$ for all $x\in \CO_P.$

The spectrum $\Spec$ of $\CO_P$ is the set of irreducible representations of $\CO_P$ modulo unitary conjugacy.
If $d$ is a natural number or $\infty$ we write $\Spec(d)$ for the $d$-dimensional part of the spectrum, i.e.~$\Spec(d)$ is the set of irreducible representations of $\CO_P$ of dimension $d$ modulo unitary conjugacy.

If $d$ is finite, then $\Spec(d)$ can be identified with the set $\Irr(d)\subset M_d(\C)\times M_d(\C)$ of pairs of matrices $(A,B)$ of size $d$ satisfying the relations \eqref{second_version_relations} modulo unitary conjugacy 
$$(A,B)\simeq W(A,B)W^*:=(WAW^*,WBW^*)$$ 
where $W\in U(d)$ is a unitary. 
The spectrum comes equipped with a topology. 
This topology coincides with the quotient topology on $\Irr(d)/\simeq$ where $\Irr(d)$ is equipped with the subset topology.
In particular, if 
$$x\mapsto (A(x),B(x))\in \Irr(d)$$ 
is a mapping that is continuous entry-wise, then it defines a continuous mapping valued in $\Spec(d).$

We will often use the following notations for diagonal and anti-diagonal matrices
$$
\antidiag(a, a') = \left(\begin{array}{cc}
0& a\\
a' & 0 \\
\end{array}\right)\text{ and }
\diag(b, b') = \left(\begin{array}{cc}
b & 0\\
0 & b' \\
\end{array}\right)
$$
where the entries may live in any suitable C*-algebra.
The following identities   will come in handy in the sequel.
As they can be readily 
verified, the proof is omitted.
\begin{enumerate}
\item $\antidiag(a,a')\antidiag(b,b')=\diag(ab',a'b)$,
\item $\diag(c,c')\antidiag(a,a')=\antidiag(ca,c'a')$,
\item $\antidiag(a,a')\diag(c,c')=\antidiag(ac',a'c).$
\end{enumerate}

\subsection{One-dimensional representations}

We begin with a complete description of one dimensional representations, which are necessarily irreducible.

\begin{proposition}
Let $\T:=\{z\in\C:\ |z|=1\}$ be the torus equipped with its usual topology.
The map 
$$\psi:\T\times \{-1,+1\}\to \C^2, \ (z,\varepsilon)\mapsto \dfrac{1}{\sqrt 2}\cdot (z,z\varepsilon)$$
defines a homeomorphism $\T\times \{-1,+1\}\to \Spec(1)$ after identifying $\Spec(1)$ with pairs of complex numbers.
\end{proposition}

\begin{proof} 
Let $\pi: \CO_P\to \IC$ be a one-dimensional representation, and set
$A:=\pi(a)$, $B:=\pi(b)\in\IC$. 
The condition of \eqref{second_version_relations} 
on $(A,\overline{A})$ and $(B, \overline{B})$ 
being P-pairs 
implies that $|A|=|B|=2^{-1/2}$.
The relation $A^2=B^2$ implies that $B\in \{\pm A\}$.
The condition on $(A, B)$ being a P-pair
imposes no further restrictions.
Conversely, observe that both $(z/\sqrt 2, z/\sqrt 2)$ and $(z/\sqrt 2,-z/\sqrt 2)$ define a one-dimensional representation of $\mathcal O_P$ for each $z\in \IT$.
All these representations are mutually inequivalent as the spectrum of
$\sigma(\pi(a))=\{A\}$ and $\sigma(\pi(b))=\{B\}$ is unitarily invariant.
Hence, $\psi$ does define a bijection 
$$\psi:\T\times\{-1,+1\}\to \Spec(1)$$ 
using the obvious identification.
Finally, observe that the topology of $\Spec(1)$ is the restriction of the topology on $\IC^2$ under our identification, see for instance \cite{Dixmier}*{Section 3.6}.
This implies that $\psi$ is continuous and is moreover closed since its domain is compact and its codomain $\Spec(1)$ is Hausdorff. 
Hence, $\psi$ is a homeomorphism.
\end{proof}

\subsection{Two-dimensional representations}
We explicitly define representatives of the 2-dimensional spectrum of $\CO_P$.

\begin{proposition}\label{prop22}
Consider the compact space $X_{00}:=[0,1/\sqrt 2]\times \IT^2$ and the equivalence relation $\sim$ determined by 
$$(0,u,v)\sim (0,u,\hat v) \text{ and } (1/\sqrt{2}, u,v)\sim (1/\sqrt{2}, \overline{u}v^2, v)$$ 
for all $u,v,\hat v\in\IT$.
Let $X_0$ be the complement of the two circles $C_\pm:=\{(1/\sqrt 2,u,\pm u): u\in\IT\}$ inside $X_{00}.$
The map 
$$\varphi_{00}:X_{00}\to M_2(\C)^2,\  (p,u,v)\mapsto \left(\begin{pmatrix} 0 & u\sqrt{1-p^2}\\ v\overline u p & 0 \end{pmatrix} , \begin{pmatrix}   0 & vp\\ \sqrt{1-p^2} & 0 \end{pmatrix}\right)$$
takes values in the representations of $\CO_P$. 
Moreover, $\varphi_{00}(p,u,v)$ is irreducible if and only if $(p,u,v)\in X_0$.
Finally, $\varphi_{00}(p,u,v)$ is unitarily equivalent to $\varphi_{00}(\hat p,\hat u,\hat v)$ if and only if $(p,u,v)\sim (\hat p, \hat u, \hat v).$
In particular, $\varphi_0$ factorises into a homeomorphism $\varphi:X\to \Spec(2)$ where $X:=X_0/\sim.$
\end{proposition}

\begin{proof}
Consider $(p,u,v)\in [0,1/\sqrt 2]\times \T\times \T$ and the pair of matrices $(A,B)=\varphi(p,u,v)$ as in the statement. 
The pair $(A,B)$ defines a representation of $\CO_P$ as
\begin{align*}
    & A^*A+B^*B=
    \begin{pmatrix}   p^2 & 0\\ 0 & 1-p^2 \end{pmatrix} +
    \begin{pmatrix}   1-p^2 & 0\\ 0 & p^2 \end{pmatrix} = 1\\
    & A^2 = 
    \begin{pmatrix}   vp\sqrt{1-p^2} & 0\\ 0 & vp\sqrt{1-p^2} \end{pmatrix} = B^2\\
    & B^*A = 
    \begin{pmatrix}   v\bar{u}p\sqrt{1-p^2} & 0\\ 0 & u\bar{v}p\sqrt{1-p^2}  \end{pmatrix} = BA^*
\end{align*}

\textbf{Claim 1: $(A,B)$ defines an irreducible representation of $\CO_P$ exactly when $(p,u,v)\in X_0$.}

As observed above the pair $(A,B)$ satisfies the relations of the generators of $\CO_P$ and thus defines by universality a representation of $\CO_P$ under the unique morphism $\CO_P\to B(H)$ sending $a$ to $A$ and $b$ to $B$.
Hence, $(A,B)$ defines a 2-dimensional (matrix) representation of $\CO_P$.
This representation is irreducible if and only if $A,B$ generate $M_2(\IC)$ as a von Neumann algebra.
Since $M_2(\C)$ does not contain any proper noncommutative von Neumann subalgebra it is sufficient to check that $AB\neq BA$ to deduce that the representation is irreducible.
Now, 
$$AB=\Delta(u(1-p^2), v^2\overline u p^2) \text{ and } BA=\Delta(v^2\overline u p^2 , u (1-p^2)).$$
If $AB=BA$, then by comparing the moduli of the entries we deduce that $p=1/\sqrt 2$ and by comparing phases we deduce that $u^2=v^2.$
Therefore, $AB\neq BA$ unless $(p,u,v)$ is in $C_+\cup C_-=\{(1/\sqrt 2,u,\pm u): u\in\T\}$.

\textbf{Claim 2: two triples of $X_0$ that are not equivalent under $\sim$ define two representations of $\CO_P$ that are pairwise inequivalent.}

Consider $(p_0,u_0,v_0)$ and $(p_1,u_1,v_1)$ in $X_0$ with associated pairs of matrices $(A_0,B_0)$ and $(A_1,B_1)$, respectively.
Assume that there exists a unitary $W\in U(2)$ satisfying $WA_0W^*=A_1$ and $WB_0W^*=B_1.$
The spectrum of $A_0^*A_0$ is $\{p_0^2,1-p_0^2\}$ implying that $p_0=p_1$.

Assume that $p_0=0$ and thus 
$$(A_0,B_0)=(u_0 e_{12}, e_{21})\qquad  \text{ and } \qquad (A_1,B_1)=(u_1 e_{12}, e_{21})$$ 
where $e_{ij}$ is the matrix with $1$ at the $(i,j)$th spot and zero elsewhere.
Hence, the $v$-parameter disappears. Moreover, 
if $(A_0,B_0)\sim (A_1,B_1)$, then $A_0+B_0\sim A_1+B_1$ and thus 
$$-u_0=\det(A_0+B_0)=\det(A_1+B_1)=-u_1.$$
Therefore, $\varphi_{00}(0,u,v)$ is equivalent to $\varphi_{00}(0,u',v')$ if and only if $u=u'$.

Assume now that $0<p_0<1/\sqrt{2}$. 
Note that $$A_0^*A_0=A_1^*A_1=\Delta(p_0^2,1-p_0^2)$$ and  $A_0^*A_0=WA_1^*A_1W^*$. 
Since $p_0^2\neq 1-p_0^2$, we deduce that $W$ is diagonal. Moreover, we may choose $W$ up to a phase. It is thus of the form $\Delta(w,1)$ with $w\in \IT.$
Since the bottom-left coefficient of $WB_0W^*=B_1$ is positive real, we deduce that $w=1.$
Therefore, the representations given by $\varphi_{00}(p,u,v)$ along the parameters $0<p<1/\sqrt 2$, $u,v\in\IT$ are pairwise inequivalent.

We treat the final case $p_0=1/\sqrt{2}$. 
We set some notations:
$$W=\begin{pmatrix}
    a & b\\ c& d
\end{pmatrix}\in U(2) \ , \  
\sqrt{2}A_0=U_0=\nabla(u_0, \overline{u_0}v_0) \ , \ 
\sqrt{2}B_0=V_0=\nabla(v_0, 1)$$
and
$$\sqrt{2}A_1=U_1=\nabla(u_1, \overline{u_1}v_1) \ , \ 
\sqrt{2}B_1=V_1=\nabla(v_1, 1)\in U(2).$$
Since $\det V_0 =-v_0$ and $\det V_1 =-v_1$, 
the condition $WV_0=V_1W$ forces
$$
v_0=v_1=:v \qquad \text{ and }\qquad 
W=\begin{pmatrix}
    a & cv \\ c & a
\end{pmatrix}.$$
The second condition $WU_0=U_1W$ reads  
$$
 \begin{pmatrix} 
cv^2\overline{u_0} & au_0\\
av\overline{u_0} & cu_0
\end{pmatrix}   =
 \begin{pmatrix} 
cu_1 & au_1 \\
av\overline{u_1} & cv^2\overline{u_1}
\end{pmatrix}  
$$
which is equivalent to having
$$a(u_0-u_1) =0\text{ and } c(u_0-v^2\overline{u_1})=0.$$
If $c=0$, then $W=a\id$ is central implying that $(A_0,B_0)=(A_1,B_1)$ and thus $(u_0,v_0)=(u_1,v_1)$.
If $c\neq 0$, then $u_0=v^2\overline{u_1}$ and thus $(u_0,v_0)=(\overline{u_1}v_1,v_1)$.
We have proven that if $\varphi_{00}(1/\sqrt 2, u_0,v_0)$ is unitarily equivalent to $\varphi_{00}(1/\sqrt 2,u_1,v_1)$, then necessarily $(u_0,v_0)=(u_1,v_1)$ or $(u_0,v_0)=(\overline{u_1}v_1,v_1)$.
Conversely, by conjugating by $W:=\nabla(v,1)$ we deduce that $\varphi_{00}(1/\sqrt 2, u,v)$ is unitary conjugate to $\varphi_{00}(1/\sqrt 2,v^2\overline u , v)$. 
It is interesting to note that we are always  in the case $ac=0$, i.e.~$W$ is diagonal or antidiagonal. Indeed, otherwise we obtain $u_0=u_1$ and $u_0=v^2\overline{u_1}$ implying that $u_0^2=v^2$. This provides $(A_0,B_0)=(A_1,B_1)$ and $(1/\sqrt 2,u_0,v_0)$ is in one of the circles $C_\pm$, a contradiction.

We have proven that $\varphi_{0}(p_0,u_0,v_0)$ and $\varphi_{0}(p_1,u_1,v_1)$ are unitarily equivalent if and only if $(p_0,u_0,v_0)\sim (p_1,u_1,v_1)$ for two triples of $X_{0}.$

\textbf{Claim 3: the map $\varphi$ is surjective.}

Let $(A,B)$ be an irreducible 2-dimensional matrix representation of $\CO_P.$
Write $A=U|A|$ and $B=V|B|$ some polar decomposition.
Since $A^*A+AA^*=1$ we obtain that, up to conjugacy, $A^*A=\Delta(p^2,1-p^2)$ for some $p\in [0,1/\sqrt 2].$
Assume for now that $p\neq 1/\sqrt 2$. This implies that $p^2\neq 1-p^2$.
Since $U\Delta(p^2,1-p^2)U^*=\Delta(1-p^2,p^2)$ we deduce that $U$ is necessarily anti-diagonal and thus of the form $\nabla(u_1,u_2)$ where $u_1,u_2\in\T.$
Hence, $A=\nabla(u_1p',u_2p)$ and similarly $B=\nabla(v_1p,v_2 p')$ for some $v_1,v_2\in\T$ where $p':=\sqrt{1-p^2}.$
The equation $A^2=B^2$ implies $u_1u_2=v_1v_2.$
Then conjugating by $\Delta(v_2,1)$ we may set $v_2$ equal to $1.$
Then necessarily $u_2=v_1\overline{u_1}$ that is $(A,B)\sim\varphi_{00}(p,u_1,v_1)$ where $\sim$ means unitary conjugate.

Consider now the case $p=1/\sqrt 2.$
We have that $(A,B)=\dfrac{1}{\sqrt 2}(U,V)$ where $U,V$ are unitaries satisfying $U^2=V^2$ and   $UV\neq VU.$
Up to unitary conjugacy, $U=\Delta(u_1,u_2)$ and thus 
$$V^2=
\left(\begin{array}{cc}
v_{11}^2+v_{12}v_{21} & (v_{11}+v_{22})v_{12}\\
 (v_{11}+v_{22})v_{21} & v_{22}^2+v_{12}v_{21} \\
 \end{array}\right)=
\left(\begin{array}{cc}
u_{1}^2  & 0\\
0 & u_{2}^2 \\
\end{array}\right)=U^2,$$
where $V=(v_{ij})_{ij}$.
Since $UV\neq VU$ the matrix $V$ is not diagonal, hence $v_{21}$ or $v_{12}$ is non-zero and thus $v_{11}+v_{22}=0$.
This implies that $u_1^2=u_2^2$.
Therefore, $U=\Delta(u,-u)$ for some $u\in\T$ (the case $\Delta(u,u)$ is excluded since $UV\neq VU$).
Set $v_{11}=zt$ with $t$ real and $z\in\T$ and thus $v_{22}=-zt.$
Since $V$ is non-diagonal and unitary we get $t\in (-1,1).$
Then $$v_{12}=\sqrt{1-t^2}\ w$$ for some $w\in\T$.
Being unitary implies the column vectors of $V$ are orthogonal giving the condition $$\overline{v_{21}}zt=\sqrt{1-t^2} w \overline z t.$$
Hence, $v_{21}=\sqrt{1-t^2} \overline{w}{z}^2$.
Moreover, $$u^2=(V^2)_{11}=v_{11}^2+v_{12}v_{21}=t^2z^2+ \sqrt{1-t^2}w \sqrt{1-t^2}\overline{w} z^2 =z^2.$$
Therefore, $z=\pm u.$
We deduce that $U,V$ are of the form
$$U=\Delta(u,-u) \text{ and } V=\begin{pmatrix}
\pm ut & w\sqrt{1-t^2}\\
\overline{w}u^2\sqrt{1-t^2} & \mp ut
\end{pmatrix}$$
where $u,w\in\T$ and $t\in (-1,1).$
By conjugating by $\Delta(1,w)$ we do not change $U$ but set the parameter $w$ inside $V$ equal to $1$.
We then conjugate by $\dfrac{1}{\sqrt{2}}\begin{pmatrix}
iu & 1\\
i u & -1
\end{pmatrix}$ obtaining the matrices 
$$\nabla(u,u) \text{ and } \nabla(u\overline v, uv) \text{ where } v=t+i\sqrt{1-t^2}.$$
Finally, by conjugating by $\Delta(uv,1)$ we obtain $\sqrt 2\cdot \varphi_{00}(1/\sqrt 2,u^2v,u^2)$.
This proves that any 2-dimensional irreducible matrix representation of $\CO_P$ is 
unitary conjugated to a certain $\varphi_{00}(p,u,v).$

\textbf{Claim 4: the map $\varphi:X\to\Spec(2)$ is a homeomorphism.}

Consider $\varphi_{00}:X_{00}\to M_2(\C)^2$ as defined in the statement of the proposition.
Note that $\varphi_{00}(p,u,v)$ is irreducible unless $p=1/\sqrt 2$ and $v=\pm u.$
Hence, $\varphi_{0}:X_0\to M_2(\C)^2$ takes values in the 2-dimensional \emph{irreducible} matrix representations $\Irr(2)$ of $\CO_P$ (under the obvious identification).
We have proven that $\varphi_0$ is surjective. 
Let
\[
\pi:X_0\longrightarrow X
\qquad\text{and}\qquad
q:\Irr(2)\longrightarrow\Spec(2)
\]
denote the quotient maps. 
By \cite{Dixmier}*{Theorem 3.5.8}, the topology of $\Spec(2)$ coincides with the quotient topology of $\Irr(2)/\simeq$ where $\simeq$ denotes the usual unitary conjugacy. 
This implies that the mapping $q$ is \emph{closed}.
We have shown that $\varphi$ is a well-defined bijection. Moreover,
$q\circ\varphi_0$ is continuous, and hence $\varphi$ is continuous by the definition of the quotient topology on $X$.

It remains to prove that $\varphi$ is closed.
Consider 
\[
\Theta_{00}:X_{00}\times U(2)\longrightarrow M_2(\C)^2,\qquad
\Theta(p,u,v,W)=W\varphi_{00}(p,u,v)W^*
\]
which is continuous and closed since it has compact domain and Hausdorff codomain.
It restricts to a map $\Theta_0:X_{0}\times U(2)\to \Irr(2)$ which is now surjective on the codomain $\Irr(2)$.
Take $Y\subset X_0\times U(2)$ closed and consider a sequence $\Theta_0(p_n,u_n,v_n,W_n)$ in $\Theta_0(Y)$ converging to some $(A,B)\in\Irr(2).$
Up to taking a subsequence, $(p_n,u_n,v_n,W_n)$ converges to some $(p',u',v',W')$ inside $X_{00}\times U(2)$ by compactness.
The continuity of $\Theta_{00}$ implies that $\Theta_{00}(p',u',v',W')=(A,B).$ 
By the above characterization of irreducible representations we deduce that $(p',u',v')\in X_0$ (rather than being in $X_{00}$).
Hence, $\Theta_0(Y)$ is closed and thus the map $\Theta_0$ is closed.

Let's take $Z\subset X$ closed.
Observe that 
$$\varphi(Z)=q(\Theta_0(\pi^{-1}(Z)\times U(2)))$$
which is closed because $q$ is a quotient, $\pi$ is continuous and $\Theta_0$ is closed.
Hence, $\varphi$ is closed and thus is a homeomorphism.
\end{proof}

\subsection{Higher dimensional representations are reducible}

Consider a representation $(A,B,H)$ of $\CO_P$ with $\dim(H)\geq 3.$
Let us show that $H$ is reducible. 
Assume for the sake of contradiction that $H$ is irreducible.
Since $\CO_P$ is separable we may assume that $H$ is separable.
Write $A=U|A|$ and $B:=V|B|$ for some polar decompositions where $|A|:=\sqrt{A^*A},|B|:=\sqrt{B^*B}$ and $U,V$ are partial isometries.
Having that $(A,A^*)$ is a P-pair means 
$$1=|A|^2+U|A|^2U^*.$$
This implies that $|A|$ has its spectrum contained in $[0,1]$ (and similarly for $|B|$). 

\begin{lemma}
The partial isometries $U,V$ can be chosen to be \emph{unitaries}.    
\end{lemma}
\begin{proof}
We will show that $A,A^*,B,B^*$ have their kernels with same dimension.
This implies the lemma via the standard result: $U$ can be chosen to be unitary if and only if the kernels of $A$ and $A^*$ have same dimension, see for instance \cite{Pedersen}*{Remark 3.2.18}. 
For the sake of clarity we divide the proof into a series of steps.
We will repeatedly make use of 
$$BB^*=A^*A \text{ and } B^*B=AA^*.$$

\textbf{It holds $\ker (A)=\ker (B^*)$, $\ker (B)=\ker (A^*)$.} Indeed, this holds regardless of the irreducibility hypothesis and follows from the following chain of equalities
\begin{align*}
&\| Av\|^2 =\langle A^*Av,v\rangle = \langle BB^*v,v\rangle = \| B^*v\|^2,\\
&\| Bv\|^2 =\langle B^*Bv,v\rangle = \langle AA^*v,v\rangle = \| A^*v\|^2.
\end{align*}

\textbf{The subspace $\ker (A^2)=\ker (B^2)$ is a subrepresentation of $\CO_P$.} Indeed, take $v\in \ker (B^2)$. Clearly $Bv$ is in $\ker (B^2)$. Moreover, 
\begin{align*}
B^2(B^*v)= & BBB^*v= Bv-BB^*Bv=Bv-Bv+B^*B^2v=0.
\end{align*}
The case $Av$ and $A^*v$ is similar. 

By irreducibility, at this point we know that either
$\ker (A^2)=\ker (B^2)=0$ or
$A^2=B^2=0$. In the former case, $A$, $A^*$, $B$, $B^*$ are all injective. 
For the remainder of the proof we work under the assumption $A^2=B^2=0$.

\textbf{If $A^2=B^2=0$, then $A^*B=0$ and $AB^*=0$.} 
Let us show that $H$ is in the kernel of $A^*B$. We know that the completion of the range of $B^*$ and the kernel of $B$ span $H$. 
Hence, it is sufficient to show that $\ker(B)\subset \ker(A^*B)$ and $\Im(B^*)\subset \ker(A^*B).$
The first inclusion is obvious. 
Take $B^*(v)$ in $\Im(B^*)$ and observe that $A^*BB^*v=A^*A^*Av=(A^2)^*Av=0.$
Hence, $A^*B=0$ and a similar proof shows that $AB^*=0$.

\textbf{If $A^2=B^2=0$, then $\dim \ker (A) = \dim\ker (A^*)$}. 
Using that $(A,B)$ and $(A,A^*)$ are P-pairs we deduce that the restriction of $B$ to $\ker(A)$ and $B^*$ to $\ker(A^*)$ are isometric.
Moreover, the previous paragraph implies that they are valued in $\ker(A^*)$ and in $\ker(A)$, respectively.
This forces  $\dim\ker(A)=\dim\ker(A^*).$
The first step proved that $\ker(A)=\ker(B^*)$ and $\ker(A^*)=\ker(B)$ and thus the proof is complete.
\end{proof}

We will now show that the spectrum of $|A|$ is either the singleton $\{1/\sqrt 2\}$ or made of two points $\{t,\sqrt{1-t^2}\}$ for some $0\leq t<1/\sqrt 2$.
Define $C:=A^*A=|A|^2$ and observe that 
$$C=1-AA^*=1-B^*B=BB^*.$$
Since $A=U|A|$ and $B=V|B|$ we deduce that
$$1-C=1-A^*A=AA^*=U|A| (U|A|)^*=U|A|^2U^*=UCU^*=:\ad(U)C$$
and similarly
$$VCV^*=V(1-B^*B)V^*=1-V|B|(V|B|)^*=1-BB^*=1-C.$$
Hence, 
$$\ad(U)(C(1-C))=(1-C)C=C(1-C)$$ and likewise $\ad(V)(C(1-C))=C(1-C).$ 
The element $C(1-C)$ is a polynomial of $|A|$ and thus commutes with $|A|$ and as well with $|B|$ since $|B|=\sqrt{1-|A|^2}.$
Since $C(1-C)$ is self-adjoint we deduce that it commutes with the *-algebra generated by $A$ and $B$ and thus must be of the form $\lambda \id_H$ by irreducibility of the representation $(A,B,H).$
This implies that if $x$ is a spectral value  of $C$, then $x(1-x)=\lambda.$
Hence, $C$ has at most two spectral values.
Using again that $UCU^*=1-C$ we deduce that either the spectrum of $C$ is $\{1/2\}$ or $\{s,1-s\}$ for some $0\leq s<1/2.$
Since $|A|^2=C$ we deduce that the spectrum of $|A|$ is either $\{1/\sqrt 2\}$ or $\{t,\sqrt{1-t^2}\}$ for some $0\leq t<1/\sqrt 2$. 
In particular, the spectrum of $|A|$ is finite which implies that the spectral values of $|A|$ are all \emph{eigenvalues}.
 
Assume that $|A|$ has a unique spectral value which must be $1/\sqrt 2$.
This means that $A=U/\sqrt 2$ and $B=V/\sqrt 2$ for some unitaries $U,V.$
The assumption $A^2=B^2$ provides $U^2=V^2.$
Then $U^2$ commutes  with both  $U$ and $V$. Irreducibility implies that $U^2=\lambda 1$ for some $\lambda\in\IC$. Take any square root $\sqrt{\lambda}$.
The operators $\tilde{U}:=\sqrt{\lambda}\ U$, $\tilde{V}:=\sqrt{\lambda}\ V$ generate the same algebra but satisfy the simpler relation $\tilde{U}^2=\tilde{V}^2$. This means that there is no loss in generality assuming $\lambda=1$.
Therefore, $U,V$ are symmetries and up to taking $(U+1)/2$ and $(V+1)/2$ we are reduced to show that two projections do not generate $B(H)$ (where, as recalled, $\dim(H)\geq 3$).
This follows from a classical result of Halmos, see \cite{Halmos69}*{p. 386}.

Assume we are in the second case where $|A|$ has for spectrum $\{ t,\sqrt{1-t^2} \}$ for a fixed $0\leq t<1/\sqrt 2$. 
For convenience we write $t':=\sqrt{1-t^2}.$
The Hilbert space $H$ decomposes as $H_t\oplus H_{t'}$ where $|A|\xi=t\xi$ and $|B|\xi=t'\xi$ for all $\xi\in H_t$ and $|A|\eta=t'\eta, |B|\eta=t\eta$ for all $\eta\in H_{t'}.$
Moreover, $U(H_t)=V(H_t)=H_{t'}, U(H_{t'})=V(H_{t'})=H_t$.
Write $K$ for a copy of $H_t$ (and thus $H_{t'}$).
We may then identify $B(H)$ with $M_2(\C)\otimes B(K)$ and interpret $|A|, |B|, U, V$ as matrices of the form:
$$|A|=\begin{pmatrix}t & 0 \\ 0 & t'\end{pmatrix}, |B|=\begin{pmatrix}t' & 0 \\ 0 & t\end{pmatrix}, U=\begin{pmatrix} 0 & W \\ X & 0\end{pmatrix}, V=\begin{pmatrix} 0 & Y \\ Z & 0\end{pmatrix}$$
where $W,X,Y,Z$ are unitaries of $K$ and where $t,t'$ stand for the identity of $K$ times $t,t'$, respectively.
Up to conjugating by $\Delta(W^*,1)$, we may assume that $W$ is the identity of $K$ in the definition of $U$. Indeed
\begin{align*}
    \Delta(W^*,1)\, U\, \Delta(W^*,1)^* &=\begin{pmatrix} W^* & 0 \\ 0 & 1\end{pmatrix}
    \begin{pmatrix} 0 & W \\ X & 0\end{pmatrix}
    \begin{pmatrix} W & 0 \\ 0 & 1\end{pmatrix} \\
    &=    
    \begin{pmatrix} 0 & 1 \\ X & 0\end{pmatrix}
    \begin{pmatrix} W & 0 \\ 0 & 1\end{pmatrix} =    \begin{pmatrix} 0 & 1 \\ XW & 0\end{pmatrix}
\end{align*}

The equation $A^2=B^2$ translates into 
$$tt'X=tt'YZ \text{ and } tt'X=tt'ZY.$$
If $t\neq 0$, then $X=YZ=ZY.$
Hence, the von Neumann algebra $N$ generated by $X,Y,Z$ is commutative and is thus a proper subalgebra of $B(K)$ since $\dim(K)>1$ by assumption.
Note that the von Neumann algebra generated by $A,B$ is contained in $M_2(\C)\otimes N$.
Therefore, $A,B$ do not generate $B(H)\simeq M_2(\C)\otimes B(K)$ as a von Neumann algebra and thus $(A,B,H)$ is a reducible representation.

Assume now that $t=0.$
We deduce that $A=\nabla(1,0)$ and $B=\nabla(0,Z)$ and thus the von Neumann algebra generated by $A$ and $B$ is contained in $M_2(\IC)\otimes M$, where $M$ is generated by the single unitary $Z$. 
This von Neumann algebra is again proper inside $B(H)$. 
Hence, the representation $(A,B,H)$ is reducible.

We have proven that any representation $H$ of $\CO_P$ of dimension at least $3$ is reducible.

\subsection{Computation of the spectrum}\label{sec:proof-main}

We may now easily deduce our main theorem.

\begin{proof}[Proof of Theorem \ref{maintheo}]
We have seen that the map 
\begin{align*}
     \psi: \IT\times\{-1,+1\}&\to\Spec(1)\\
     (u,\varepsilon)&\mapsto \dfrac{1}{\sqrt 2}(u,\varepsilon u)
\end{align*} 
defines a homeomorphism. 
Consider $X_{00}=[0,1/\sqrt 2]\times \T^2$ and the map 
\begin{align*}
\varphi_{00}:X_{00}&\to M_2(\IC)^2\\
(p,u,v)&\mapsto \left(\begin{pmatrix} 0 & u\sqrt{1-p^2}\\ v\overline u p & 0 \end{pmatrix} , \begin{pmatrix}   0 & vp\\ \sqrt{1-p^2} & 0 \end{pmatrix}\right).
\end{align*} 
Recall that $X_0=X_{00}\setminus(C_+\cup C_-)$ where $C_{\pm}:=\{(1/\sqrt 2, u,\pm u):\ u\in\IT\}$ and $X:=X_0/\sim$ where $(0,u,v)\sim (0,u,\hat v)$ and $(1/\sqrt 2, u,v)\sim(1/\sqrt 2, \overline u v^2,v)$ for all $u,v,\hat v\in\IT.$

We have seen that $\varphi_{00}(p,u,v)\in \Irr(2)$ (i.e.~defines an irreducible representation of $\CO_P$) unless $(p,u,v)\in C_+\cup C_-$. Moreover, two such representations are inequivalent unless they belong to the same $\sim$-class.
Moreover, $\varphi_{00}$ factorises into a homeomorphism $\varphi:X\to\Spec(2).$

We have proven that if $(A,B,H)$ is a representation of $\CO_P$ so that $\dim(H)>2$, then $H$ is reducible.
Hence, $\Spec=\Spec(1)\sqcup \Spec(2)$ as a set (but not as a topological space).

We do have that $\Spec(1)$ and $\Spec(2)$ are individually Hausdorff.
However, the whole spectrum $\Spec$ is not.
Indeed, consider the point $(1/\sqrt 2, u^2,u^2)$ which is in $X_{00}$ but is not in $X_0.$
It defines a reducible representation of $\CO_P$ of dimension 2.
It is unitary equivalent to the direct sum of two characters (i.e.~representations of dimension one) that are $\psi(u,+1)$ and $\psi(-u,+1).$
This implies that if $\Omega,\Omega'$ are neighbourhoods of $\psi(u,+1),\psi(-u,+1)$, respectively, then $\Omega\cap\Omega'$ always contains an element of $\Spec(2)$ of the form $\varphi(p,u^2,u^2)$ with $p$ sufficiently close to $1/\sqrt 2.$
Hence, the two points $\psi(u,+1)$ and $\psi(-u,+1)$ are not separated in $\Spec$ even if they are separated in $\Spec(1).$
A similar method shows that the net $\varphi(p,u^2,-u^2)\in\Spec$ (with $0<p<1/\sqrt 2)$ has   limit points $\psi(iu,-1)$ and $\psi(-iu,-1)$ when $p$ tends to $1/\sqrt 2$. Hence, $\psi(iu,-1)$ and $\psi(-iu,-1)$ are not separated inside $\Spec.$
\end{proof}

\subsection{Representations of Richard Thompson's groups and of the Cuntz algebra}\label{sec:FTVO}

We relate representations of $\CO_P$ with representations of the Thompson groups $F,T,V$ and representations of the (binary) Cuntz algebra $\CO_2$ \cite{CFP, Cuntz}.
Jones and the second named author have defined a functorial process that takes 
as inpput a pair of operators $A,B$ satisfying that $A^*A+B^*B=1$, and gives as output a representation $\pi_{A,B}$ of the Cuntz algebra $\CO_2$ with two generators \cite{Brothier-Jones19}. 
Using the Birget-Nekrashevych embedding of the three Thompson groups $F\subset T\subset V$ inside the unitary group of $\CO_2$ we may then restrict $\pi_{A,B}$ and obtain unitary representations of $F,T,V$ \cite{Birget03,Nekrashevych03}. 
Hence, any representation of the so-called Pythagorean algebra 
$$\CP_2:=C^*\langle a,b| a^*a+b^*b=1\rangle$$
yields a representation of $\CO_2$ and then unitary representations of $F,T,V.$
Now, observe that our Cuntz-Pimsner algebra $\CO_P$ is a quotient of $\CP_2$ (under the obvious mapping).
Hence, any representation $(A,B,H)$ of $\CO_P$ produces a representation $\pi_{A,B}$ of $\CO_2$ which restricts to unitary representations of $F,T,V$.
We wonder what kind of representations are obtained.

In a series of articles, Wijesena and the second named author have continued to study the procedure $(A,B,H)\mapsto \pi_{A,B}$ \cite{Brothier-Wijesena25, Brothier-Wijesena26, Brothier-Wijesena24}.
In general, this procedure preserves neither  irreducibility nor unitary conjugacy.
However, if we moreover require that $H$ is finite dimensional and does not contain any subspace that is
closed under $A$ and $B$ (which is stronger than asking irreducibility of $H$ since we do not consider the adjoints $A^*,B^*$), then $\pi_{A,B}$ is an irreducible representation of $\CO_2$.
Under these assumptions, two such representations $(A,B,H)$ and $(A',B',H')$ are unitary conjugate if and only if the associated $\pi_{A,B}$ and $\pi_{A',B'}$ are unitary conjugate.
Additionally, the restrictions of $\pi_{A,B}$ to the Thompson groups $F,T,V$ are all irreducible unless in the specific case: $\dim(H)=1$ and $A$ or $B$ is zero.
In dimension one, the representations of $\CP_2$ are all irreducible and their classes are in bijection with the real 3-sphere $\{(x,y)\in\C^2:\ |x|^2+|y|^2=1\}.$
In dimension two, it was proven that the set of classes of $(A,B,H)$ with the
above condition  was in bijection with a real smooth manifold of dimension $9$. 
This manifold sits inside the spectrum of $\CO_2$ via the procedure $(A,B,H)\mapsto \pi_{A,B}.$

The C*-algebra $\CO_P$ admits the pair $(1/\sqrt 2,1/\sqrt 2)$
as a representation acting on $\IC$. 
The restriction to $F$ (resp.~$T,V$) of $\Pi(1/\sqrt 2,1/\sqrt 2)$ provides the Koopman representation of the usual action of $F$ on $[0,1]$ (resp.~$T,V$).
Now, if $u$ has modulus $1$, then the representation $(u/\sqrt 2,u/\sqrt 2)$ of $\CO_P$ corresponds to the $u$-deformation of the Koopman representation obtained by Garncarek in \cite{Garncarek12}.
Hence, $\Spec(1)$ of $\CO_P$ corresponds to the one circle deformation of the Koopman representation of $F$ and a second circle. 
These two circles live in the 3-sphere $S^3.$
These representations are diffuse in the sense of \cite{Brothier-Wijesena25}.
This means that the representations of $F,T,V$ are Ind-mixing: they are never conjugate to the induced representation of a finite dimensional representation. 
In particular, they are weakly mixing but note that they are never mixing.

Consider now an irreducible representation $(A,B)$ of $\CO_P$. 
We start by showing that not only does $(A,B,A^*,B^*)$ not fix a line (which is equivalent to having an irreducible representation of $\CO_P$) but moreover the pair $(A,B)$ does not fix a line.
This will imply that the associated representation of the usual Cuntz algebra $\CO_2$ is irreducible and so are its restrictions to $F,T,V$.
The dimension 1 case is clear. We prove it in detail in dimension 2.
\begin{lemma}
Let $\varphi_{0}(p,u,v)=(A,B)$ be a 2-dimensional irreducible representation, then
$\IC^2$ does not admit any line that is preserved by $A$ and $B$.    
\end{lemma}
\begin{proof}
The existence of a line preserved by both $A$ and $B$ amounts to these matrices having a common eigenvector. This does not depend on the equivalence class of the representation, so we may assume that $(A, B)$ is of the form 
$\varphi_0(p,u,v)=(A,B)$, with $(p,u,v)\in X_0.$
We divide the proof into several cases. 

\textbf{Case $p=0$}. In this case $A=ue_{12}$, $B=e_{21}$.
The only invariant line of $A$ (resp.~$B$) is $\IC e_2$ (resp.~$\IC e_1$). Hence, $A$ and $B$ do not fix a common line.

\textbf{Case $0<p\leq 1/\sqrt 2$}.
Set $q=\sqrt{1-p^2}$. Recall that
\[
A=\begin{pmatrix}
0 & uq\\
v\overline u\,p & 0
\end{pmatrix},
\qquad
B=\begin{pmatrix}
0 & vp\\
q & 0
\end{pmatrix}.
\]
Assume that $L:=\IC \xi$ is line invariant under $A$ and $B$.
We begin by observing that if $\xi=xe_1+ye_2$, then both $x$ and $y$ must be non-zero because $Ae_1\in\IC e_2$ and
$Ae_2\in\IC e_1$.  
Therefore, we may assume $x=1$.

The invariance of $L$ under $A$ and $B$ implies that there exist
$\lambda,\mu\in\IC$ satisfying
\begin{align*}
&A \begin{pmatrix}1\\ y\end{pmatrix} = \begin{pmatrix} uqy \\ v\bar{u}p \end{pmatrix} = \lambda \begin{pmatrix}1\\ y\end{pmatrix}= \begin{pmatrix}1\\ \lambda y\end{pmatrix}\\
&B\begin{pmatrix}1\\ y\end{pmatrix}=\begin{pmatrix}vpy\\ q\end{pmatrix}
=\mu \begin{pmatrix}1\\ y\end{pmatrix} = \begin{pmatrix}\mu \\ \mu  y\end{pmatrix}.
\end{align*}
Since $p,q,y$ are all nonzero we deduce that $\lambda\neq 0\neq \mu.$
By taking the product of the entries of $A\xi$ (resp.~of $B\xi$)
we deduce that
$$pqvy=\lambda^2y \text{ and } pqvy=\mu^2 y.$$
Since $y\neq 0$ we deduce that $\lambda^2=\mu^2$ and thus $\lambda=\mu$ or $\lambda=-\mu$.
Taking the moduli of the first entries of $A\xi,B\xi$ we deduce that $p=q$.
Hence, we must have $p=q=1/\sqrt 2$.
By comparing the the first entries of $A\xi$ and $B\xi$ we deduce that $u=v$ or $u=-v.$
Hence, $(p,u,v)$ is in one the circles $C_\pm$ of Proposition \ref{prop22}.
We have reached a contradiction since by assumption $(p,u,v)\in X_0.$
\end{proof}

This analysis implies that $\pi_{A,B}\in\Rep(\CO_2)$ is irreducible and also the restrictions to $F,T,V$ are irreducible.
When $p\neq 0$, all these representations are diffuse. Hence, their restrictions to $F,T,V$ are Ind-mixing.
Now, if $p=0$ and $u=1$, then we recognise the representation of $\CP_2$ appearing in \cite[Section 6.1]{Brothier-Jones19}.
Hence, the representation $\pi_{A,B}$, restricted to $F$, is equivalent to the quasi-regular representation associated to the parabolic subgroup $F_{1/3}$ which fixes the point $1/3$ for the usual action $F\curvearrowright [0,1]$. 
A similar statement occurs when $p=0$ and $u\neq 1$, where the quasi-regular representation is replaced by a monomial representation associated to a character of $F_{1/3}$. 
Also, similar statements hold when considering $T,V$ rather than $F.$
Hence, the image of $\Spec(2)$ by the procedure $(A,B,H)\mapsto \pi_{A,B}$ is a piece of the spectrum of $\CO_2$ contained in a manifold of dimension 3 (living in a manifold of dimension $9$) composed of diffuse representations and a circle of atomic representations (those that restrict to monomial representations associated to $F_{1/3}$).

\section*{Acknowledgements} 
We thank the referee for their careful perusal of the manuscript, which resulted in the simplification of some proofs and improved the overall clarity of the presentation.

 V.A. is partially supported by Sapienza Universit\`a di Roma (Progetto di Ateneo Dipartimentale 2024 “New research trends in Mathematics at Castelnuovo”),
  by  INdAM-GNAMPA  
through the program "Partecipazione a Convegni, Scuole, Workshop e
Cicli di Seminari", and by the INdAM-GNAMPA  project
“Simmetrie distribuzionali per processi stocastici quantistici” CUP E53C25002010001.
V.A. thanks the University of Trieste for its hospitality during a one-week visit.
A.B. is grateful to INdAM-GNFM for its support.

\textbf{AI Statements.} Basic free versions of ChatGPT were used to check the English spelling and grammar, and to computing matrix multiplications.

\end{document}